\documentclass{amsart}
\usepackage{amssymb,amsfonts,amsmath}
\usepackage[all]{xy}
\usepackage[shortlabels]{enumitem}
\usepackage{hyperref, color}		
\newtheorem{theorem}{Theorem}[section]
\newtheorem{proposition}[theorem]{Proposition}

\newtheorem{corollary}[theorem]{Corollary}
\newtheorem{definition}{Definition}[section]
\newtheorem{lemma}[theorem]{Lemma}

\newtheorem{claim}{Claim}[section]

\theoremstyle{theoremA}
\newtheorem*{theoremA}{Theorem A}
\theoremstyle{theoremB}
\newtheorem*{theoremB}{Theorem B}  
\newcommand{\real}{\mathbb{R}}

\newcommand{\Si}{\Sigma}

\newcommand{\te}{\theta}

\newcommand{\la}{\lambda}

\newcommand{\al}{\alpha}
\newcommand{\om}{\omega}
\newcommand{\Om}{\Omega}
\newcommand{\na}{\nabla}
\newcommand{\ep}{\epsilon}
\newcommand{\ga}{\gamma}

\newcommand{\be}{\beta}

\newcommand{\De}{\Delta}
\newcommand{\de}{\delta}

\newcommand{\lan}{\langle}
\newcommand{\ran}{\rangle}

\newcommand{\p}{\partial}

\newcommand{\dv}{\mathrm{div}}
\newcommand{\m}{\mathcal}

\newcommand{\vp}{\varphi}
\newcommand{\hs}{\textrm{Hess}}
\begin{document}


\title[The positive mass Theorem and Penrose inequality for graphs]{The positive mass Theorem and Penrose inequality for graphical manifolds}
\thanks{The second author was partially supported by CNPq-Brazil}

\author{H. Mirandola}

\author{F. Vit\'orio}

\date{}

\maketitle
\thispagestyle{empty}
\begin{abstract}  
We give, via elementary methods, explicit formulas for the ADM mass which allow us to conclude the positive mass theorem and Penrose inequality for a class of graphical manifolds  which includes, for instance, that ones with flat normal bundle.
\end{abstract}
\section{Introduction}
A smooth connected $n$-dimensional Riemannian manifold $(M^n,g)$ is said to be {\it asymptotically flat} if there exists a compact subset $K$ of $M$ and a diffeomorphism $\Phi:M\setminus K\to \real^n \setminus \{|x|\leq 1\}$ such that in this coordinate chart the metric $g(x)=g_{ij}(x)dx_i\otimes dx_j$, with $x=(x_1,\ldots,x_n) \in \real^n\setminus \{|x|\leq 1\}$, satisfies
\begin{equation*}
\begin{array}{ll}
g_{ij}-\de_{ij} = O(|x|^{-p})\,, & g_{ijk}=O(|x|^{-p-1})\\
g_{ijkl} = O(|x|^{-p-2}) \,, & S= O(|x|^{-q}),
\end{array}
\end{equation*}
at infinity, where $|x|=\sqrt{x_1^2+\ldots+x_n^2}$ and $g_{ijk}$, $g_{ijkl}$ denote the partial derivatives of $g_{ij}$,
\begin{equation} 
g_{ijk}=\frac{\p g_{ij}}{\p x^k} \ \ \mbox{ and } \ \ g_{ijkl}=\frac{\p^2 g_{ij}}{\p x^k \p x^l},
\end{equation}
for all $1\leq i,j,k\leq n$. Here $S$ is the scalar curvature, $\de_{ij}$ is the Kronecker delta,  and $p>\frac{n-2}{2}$ and $q>n$ are constants.
\begin{definition}\label{defadmmass} The ADM mass of a manifold $(M,g)$ is the limit 
\begin{equation}\label{admmass}
m_{ADM} = \lim_{r\to \infty} \frac{1}{2(n-1)\om_{n-1}}\int_{S_r}(g_{iji}-g_{iij}){\nu_j} d\mu, 
\end{equation}
where $S_r =\{x\in \real^n \mid |x|=r\}$ is the coordinate sphere of radius $r$, $d\mu$ is the area element of $S_r$ in the coordinate chart, $\om_{n-1}$ is the volume of the unit sphere $S_1$ and $\nu=(\nu_1,\ldots,\nu_n)=r^{-1}x$ is the outward unit normal to $S_r$. Henceforth, all the repeated indices are being summed as usual.
\end{definition}
It is worthwhile reminder that definition \ref{defadmmass} was given \cite{adm} by the physics Arnowitt, Deser  and  Misner, for the tridimensional case, and 
Bartnik \cite{ba} proved that for an asymptotically flat manifold the limit (\ref{admmass}) exists and the definition of the ADM mass of $g$ is independent of  the choice an asymptotically flat chart $\Phi$, hence the ADM mass is an geometric invariant of $(M,g)$. The positivity of the ADM mass in all dimensions is a long-standing question and a pillar of the mathematical relativity. In a seminal work, Schoen and Yau \cite{sy1} gave an affirmative answer for the tridimensional case. In a follow-up paper, they also proved for dimensions $3\le n\le 7$ (see \cite{sy2}). For manifolds that are conformally flat or spin affirmative answers are given by Schoen-Yau \cite{sy3} and Witten \cite{wi}, respectively. The Riemannian positive mass theorem can be stated as
\begin{theoremA}[\cite{sy1,sy2,sy3,wi}] Let $M^n$ be an asymptotically flat manifold with nonnegative scalar curvature. Assume that $M$ is spin, or $3\le n\le 7$, or $M$ is conformally flat. Then the ADM mass is positive unless $M$ is isometric to  $\real^n$. 
\end{theoremA}

The Riemannian Penrose conjecture asserts that any asymptotically flat manifold $M$ with nonnegative scalar curvature containing an outermost minimal hypersurface  (possibly disconnected) of area $A$ has ADM mass satisfying 
\begin{equation}\label{admpenroseinq}
m_{ADM}\ge \frac{1}{2} \left(\frac{A}{\om_{n-1}}\right)^{\frac{n-2}{n-1}}\!\!\!\!.
\end{equation}
Furthermore, the equality in (\ref{admpenroseinq}) implies that $M$ is isometric to the Riemannian Schwarzschild manifold. We want to point out that this inequality was first proved in the three-dimensional case by Huisken and Ilmanen \cite{hi} under the additional hypothesis that $\Si$ is connected.  Bray \cite{br} proved this conjecture, still in dimension three, without connectedness assumption on $\Si$. For $3\le n\le 7$, this conjecture was proved by Bray and Lee \cite{bl}, with the extra requirement that $M$ be spin for the rigidity statement. The Riemannian Penrose inequality can be stated as 
\begin{theoremB}[\cite{hi,br,bl}] Let $M^n$ be an asymptotically flat manifold with nonnegative scalar curvature. Assume that $3\le n\le 7$ and there exists an outermost minimal hypersurface $\Si\subset M$ with area $A$. Then 
\begin{equation*}
m_{ADM}\ge \frac{1}{2} \left(\frac{A}{\om_{n-1}}\right)^{\frac{n-2}{n-1}}\!\!\!\!.
\end{equation*}
Moreover, under the hypothesis that $M$ is spin then the equality occurs if and only if $M$ is the Riemannian Schwarzschild manifold.
\end{theoremB}

Recently, Lam \cite{lam} obtained an elementary and straightforward proof of the positive mass theorem and the Penrose inequality for codimension one graphical manifolds, which
was extended in some sense to hypersurfaces by Huang and Wu in \cite{huang,huang2,huang3} and, for more general codimension one graphs, by de Lima and Gir\~ao in \cite{levi,levi2}. This paper deals with graphical manifolds with arbitrary codimensions. We give here, via elementary methods, an explicit formula for the ADM mass which allow us to conclude the positive mass theorem and Penrose inequality for a class of graphical manifolds which includes, for instance, that ones with flat normal bundle. We bring to the fore that graphical manifolds with flat normal bundle are subject of study in several recent works, see for example  \cite{jx}, \cite{mu}, \cite{swx} and references therein.

To enunciate our theorems we will start with some notations and definitions.

\begin{definition} A $C^2$ map $f:\real^n\setminus\Om\to \real^m$, where $\Om\subset \real^n$ is a subset, is said to be {\it asymptotically flat} if the scalar curvature  $S$ of the graph of $f$ endowed with the natural metric is an integrable function over $\real^n$ and moreover the partial derivatives $f_i^\al=\frac{\p f^\al}{\p x_i}$ and $f_{ij}^\al=\frac{\p^2 f^\al}{\p x_i\p x_j}$  satisfy 
\begin{equation*}
\begin{array}{ll}
|f^\al_i(x)|=O(|x|^{-\frac{p}{2}});\ \ \
|f^\al_{ij}(x)|=O(|x|^{-\frac{p}{2}-1});
\end{array}
\end{equation*}
at infinity, for all $\al=1,\ldots,m$ and $i,j,k=1,\ldots,n$, where $p>(n-2)/2$. 
\end{definition}
Let $M=\{(x,f(x))\mid x\in \real^n\}$ be the graph of an asymptotically flat map $f:\real^n\to \real^m$ endowed with the natural metric. 
%
The vectors $\p_i=(e_i, f_i^\al e_\al)$ form the coordinate vector fields and the vectors $\eta^\al=(-Df^\al,e_\al)$, where $Df^\al$ denotes the gradient vector field of $f^\al$, form a basis of the normal bundle of $M$. Here $e_i$ and $e_\al$ denotes the canonical vectors of $\real^n$ and $\real^m$, respectively. 
The natural metric $g=g_{ij}dx^i \otimes dx^j$ of $M$ is given by
\begin{equation}\label{graph-metric}
g_{ij}=\de_{ij}+f_i^\al f_j^\al,
\end{equation}
hence 
$g_{ij} = O(|x|^{-p})$ and $g_{ijk}=O(|x|^{-p-1})$.

By abuse of notation, let us consider that the functions $f^\al$ are also defined on $M$ by identifying  $f^\al=f^\al\circ \pi$, where $\pi:M\to \real^n$ is the natural projection $\pi(x,f(x))=x$, for all $x\in\real^n$. The gradient vector field of $f^\al:M\to\real$ satisfies 
\begin{equation}\label{grad-f}
\na f^\al = g^{jk}f^\al_k \p_j,
\end{equation}
where the matrix $(g^{ij})$ denotes the inverse matrix $(g_{ij})^{-1}$.

Let $S:\real^n\to \real$ be the scalar curvature of $M$ and $S^\perp:\real^n\to \real$ the function given by
\begin{equation}\label{normal-scalar}
S^\perp = \lan R^\perp(\na f^\al,\na f^\be)\eta^\be,\eta^\al \ran,
\end{equation}
where $R^\perp$ denotes the normal curvature tensor of the submanifold $M\subset \real^{n+m}$. 

In the three theorems below, we will state explicit formulas for the ADM mass. As consequence, we will derive the Riemannian positive mass and Penrose inequalities for graphical manifolds with flat normal bundle.

\begin{theorem}\label{mass-graph} Let $M^n$ be a graph of an asymptotically flat map $f:\real^n\to \real^m$ endowed with its natural metric $g=g_{ij}dx^i\otimes dx^j$.  Then the ADM mass of $M$ satisfy
\begin{equation*}
m_{ADM} = \frac{1}{2(n-1)\om_{n-1}}\int_M (S+S^\perp) \frac{1}{\sqrt G}dM,
\end{equation*} 
where $G$ is the determinant of the metric coefficient matrix $(g_{ij})$.
\end{theorem}

\begin{theorem}\label{penrose-graph} Let $\Om\subset \real^n$ be a bounded open subset  with Lipschitz boundary $\p\Om$. Let $f:\real^n\setminus \Om \to \real^m$ be an 
asymptotically flat map. 
Assume that $f$ is constant along each connected component of $\Si=\p\Om$. Let $M$ be the graph of $f$ with its natural metric. 
Then,
\begin{equation*}
m_{ADM}= \frac{1}{2(n-1)\om_{n-1}}\left(\int_M (S+S^\perp) \frac{1}{\sqrt G}dM + \int_\Si \frac{|Df|^2}{1+|Df|^2} H^\Si d\Si\right),
\end{equation*}
where $|Df|^2=|Df^1|^2+\ldots+|Df^m|^2$ and $H^\Si$ is the mean curvature of the hypersurface $\Si$ in the Euclidean space $\real^n$ in the direction to the unit vector field $\nu$ pointing outward to $\Om$. 
\end{theorem}

\begin{theorem}\label{cor-penrose-graph} Let $\Om\subset \real^n$ be an open subset. Let $f:\real^n\setminus \Om\to \real^m$ be a continuous map that is constant along each connected component of the boundary $\Si=\p\Om$ and  asymptotically flat in $\real^n\setminus\bar\Om$.  Assume  that the graph $M$ of $f$ extends $C^2$ up to its boundary $\p M$. Assume further that  $\lim_{x\to\p\Om} S^\perp=0$ and, along each connected component $\Si_i$ of $\p M$, the manifold $\bar M$ is tangent to the cylinder $\Si\times \ell_i$, where $\ell_i$ is a straight line of $\real^m$.  Then,  
\begin{equation*}
m_{ADM}= \frac{1}{2(n-1)\om_{n-1}}\left(\int_M (S+S^\perp) \frac{1}{\sqrt G}dM + \int_\Si  
H^\Si \, d\Si\right),
\end{equation*}
where $H^\Si$ is the mean curvature of the hypersurface $\Si\subset\real^n$ in the direction of the unit vector field $\nu$ pointing outward to $\Om$.
\end{theorem}

As a consequence of Theorem \ref{mass-graph} it follows the Riemannian positive mass inequality for graphs  whose normal fiber bundle is flat. More specifically, 

\begin{corollary} Let $M^n\subset \real^{n+m}$ be the graph of an asymptotically flat map $f:\real^n\to \real^m$ endowed with the natural metric. Assume that $M$ has nonnegative scalar curvature and flat normal fiber bundle. Then the ADM mass of $M$ is nonnegative. 
\end{corollary}

Now we will state a Penrose-type inequality for graphs manifolds with arbitrary codimension. Following \cite{huang3} closely, we can use the following Alexandrov-Fenchel inequality due to Guan and Li \cite{guanli} and an elementary lemma.

\begin{proposition}[\cite{guanli}] \label{guanli-af}
Let $\Om\subset \real^n$ be a star-shaped domain with boundary $\p\Om=\Si$. Then,
\begin{equation} \label{af-inequality}
\frac{1}{2(n-1)\om_{n-1}} \int_\Si H_\Si \, d\Si\geq \frac{1}{2} \left(\frac{|\Si|}{\om_{n-1}}\right)^{\frac{n-2}{n-1}},
 \end{equation}
where $H^\Si$ is the mean curvature of the hypersurface $\Si\subset\real^n$ in the direction of the unit vector field $\nu$ pointing inward to $\Om$. Furthermore, the equality in (\ref{af-inequality}) occurs if and only if $\Si$ is a sphere.
\end{proposition}

\begin{lemma}[\cite{huang3}] \label{elem-ineq}Let $a_1, \cdots, a_k$ be nonnegative real numbers and $0\le \be \le 1$. Then,
\[
\sum_{i=1}^k a_i^\be \geq \left( \sum_{i=1}^k a_i \right)^\be.
\]
If $0\le \be < 1$, the equality holds if and only if at most one of $a_i$ is non-zero.
\end{lemma}

The theorem \ref{cor-penrose-graph}, the proposition \ref{guanli-af} and the lemma \ref{elem-ineq} allow us to conclude our main result

\begin{theorem}\label{main-theorem} Under hypothesis of Theorem \ref{cor-penrose-graph}, we  assume that $M$ has non-negative scalar curvature and flat normal fiber bundle. Assume further that each connected component of $\Om$ is star-shaped. Then, \begin{equation} \label{rigidity}m_{ADM}\ge \frac{1}{2}\left(\frac{|\Si|}{\om_{n-1}}\right)^{\frac{n-2}{n-1}},\end{equation} where $|\Si|$ denotes the  total volume of $\Si$. Furthermore, the equality in (\ref{rigidity}) implies that the scalar curvature $S$ is identically zero and $\Si$ is a sphere.
\end{theorem}

Some questions arise in this paper:
\begin{itemize}
\item[a)] Can we obtain an isometric immersion theorem, in the sense of the Nash theorem, so that an asymptotically flat manifold is a graph in arbitrary codimension? If yes, is it possible in such way that the normal fiber bundle is flat? 
\item[b)]  Can we obtain the rigidity in the theorem \ref{main-theorem}? We believe that extensions of the works of Schoen \cite{s} and Hounie-Leite \cite{hl} to submanifolds can bring an answer to this question.
\end{itemize}

\section{Preliminaries}
We assume the notations in the previous section. Let $U=(U_{\al\be})$ be the non-singular matrix given by $$U_{\al\be}=\lan \eta^\al,\eta^\be\ran = \de_{\al\be}+\lan Df^\al,Df^\be\ran$$ and $(U^{\al\be})=U^{-1}$ its inverse matrix.
Using that 
$\bar\na_{\p_i}\eta^\al = \eta^\al_i = (-Df^\al_i,0)$,
we obtain that $\lan\bar\na_{\p_i} \eta^\al,\p_j \ran = -f_{ij}^\al$ and $\lan \bar\na_{\p_i}\eta^\al, \eta^\be \ran= \lan Df_i^\al,Df^\be\ran$. Thus the shape operator $A^\al$ with respect to the normal vector $\eta^\al$ and the second fundamental form $B$ satisfy
\begin{eqnarray}\label{shape-oper} 
A^\al \p_i &=& -(\bar\na_{\p_i} \eta^\al)^T = f_{ik}^\al g^{kj} \p_j;\\
\nonumber 
B(\p_i,\p_j) &=& f_{ij}^\al U^{\al\be}\eta^\be.
\nonumber 
\end{eqnarray}
By Gauss Equation, the curvature tensor $R$ of $M$ satisfies
\begin{eqnarray*}
R_{ilkj}&=&\lan R(\p_i,\p_l)\p_k,\p_j\ran = \lan B(\p_i,\p_j),B(\p_l,\p_k)\ran - \lan B(\p_i,\p_k),B(\p_l,\p_j)\ran\\
&=& f_{ij}^\ga U^{\ga\al}  f^\mu_{kl}U^{\mu\be} U_{\al\be} -  f_{ik}^\ga U^{\ga\al}  f^\mu_{jl}U^{\mu\be} U_{\al\be} = (f_{ij}^\ga f_{kl}^\al - f_{ik}^\ga f_{jl}^\al)U^{\ga\al}
\end{eqnarray*}
Thus the scalar curvature $S:\real^n\to \real$ of $M$ satisfies 
\begin{equation}\label{scalar-expression1}
S = g^{ij}g^{kl}R_{ilkj}=g^{ij}g^{kl}U^{\al\be}(f_{ij}^\be f_{kl}^\al - f_{ik}^\be f_{jl}^\al).
\end{equation}

We will prove the following
\begin{proposition}\label{scalarcurvature-divX} The scalar curvature $S:\real^n\to \real$ of the graph $M$ and the function $S^\perp=\lan R^\perp(\na f^\al,\na f^\be)\eta^\be,\eta^\al\ran$ as given in (\ref{normal-scalar}) satisfy
\begin{equation*}
S+S^\perp = \na \cdot X,
\end{equation*}
where $X:\real^n\to \real^n$ is the vector field given by
\begin{equation}\label{expression-X}
X =(U^{\al\be}(f_{i}^\be f_{kk}^\al - f_{k}^\be f_{ik}^\al)+ U^{\al\ga}U^{\be\mu}\lan Df^\ga,Df_k^\mu\ran (f_i^\al f_k^\be - f_k^\al f_i^\be))\,e_i.
\end{equation}
\end{proposition}

Before we prove Proposition \ref{scalarcurvature-divX} we will need some preliminaries. For our purposes it is convenient to write $M_{ij} = \de_{ij} - g^{ij}$.  It is simple see that 
\begin{lemma}\label{prelim1} Under the notations above, the following items hold:
\begin{enumerate}[(I)]
\item\label{useful} $f_i^\al U^{\al\be} = f^\be_j g^{ji}$;
\item\label{woodbury} $M_{ij} = f_i^\al f_k^\al g^{kj} = f_i^\al f_j^\be U^{\al\be}$;
\item\label{U-inv-metric} $g(\na f^\al,\na f^\be)=\de_{\al\be}-U^{\al\be}$.
\end{enumerate}
\end{lemma}
\begin{proof}
By (\ref{graph-metric}) we have that $$f^\be_j g_{ij}=f_j^\be(\de_{ij}+f_i^\al f_j^\al)=f^\be_i + f_i^\al \lan Df^\be,Df^\al\ran = f_i^\be + f_i^\al (U_{\al\be}-\de_{\al\be}) = f_i^\al U_{\al\be}.$$ Thus, Item\ref{useful} follows by multiplying both sides by  $g^{ik}U^{\be\mu}$.  Again using (\ref{graph-metric}) we have that $\de_{ij}=g_{ik}g^{kj}=(\de_{ik}+f_i^\al f_j^\al)g^{kj}=g^{ij} + f_i^\al f_j^\al g^{kj}$. This together Item \ref{useful} imply that $g^{ij}=\de_{ij} - f_i^\al f_k^\al g^{kj} = \de_{ij} - f_i^\al f_j^\be U^{\be\al}$, hence $M_{ij}=f_i^\al f_k^\al g^{kj}= f_i^\al f_j^\be U^{\be\al}$, hence Item \ref{woodbury} holds.  Now, by (\ref{grad-f}), we obtain 
\begin{eqnarray*}
g(\na f^\al,\na f^\be) &=& g^{ij}f_i^\al f_j^\be = (\de_{ij}-U^{\ga\mu}f_i^\ga f_j^\mu)f_i^\al f_j^\be\\ &=& \lan Df^\al,Df^\be\ran - U^{\ga\mu}\lan Df^\ga,Df^\al\ran\lan Df^\mu,Df^\be\ran\\ &=& \lan Df^\al,Df^\ga\ran(\de_{\ga\be} - U^{\ga\mu}(U_{\mu\be}-\de_{\mu\be}))=\lan Df^\al,Df^\ga\ran U^{\ga\be}\\&=&(U_{\al\ga}-\de_{\al\ga})U^{\ga\be}=\de_{\al\be}-U^{\al\be},
\end{eqnarray*}
We obtain Item \ref{U-inv-metric}. Lemma \ref{prelim1} is proved.
\end{proof} 
The following result is useful to prove Proposition \ref{scalarcurvature-divX}.
\begin{lemma}\label{prelim2} The following items are true:
\begin{enumerate}[(i)]
\item\label{dijdkl} $\de_{ij}\de_{kl} U^{\al\be}(f_{ij}^\be f_{kl}^\al - f_{ik}^\be f_{jl}^\al)=U^{\al\be}(f_i^\be f_{kk}^\al - f_k^\be f_{ik}^\al)_i$;
\item\label{dklmij} $\de_{ij}M_{kl}U^{\al\be}(f_{ij}^\be f_{kl}^\al - f_{ik}^\be f_{jl}^\al)=\de_{kl}M_{ij}U^{\al\be}(f_{ij}^\be f_{kl}^\al - f_{ik}^\be f_{jl}^\al);$
\item\label{2dijmkl} $2\de_{ij}M_{kl}U^{\al\be}(f_{ij}^\be f_{kl}^\al - f_{ik}^\be f_{jl}^\al)=-U^{\al\be}_i (f_i^\be f_{kk}^\al -  f_k^\be f_{ik}^\al) - U^{\al\ga}U^{\be\mu}\lan Df^\ga,Df^\mu_i\ran F^{\be\al}_{ik,k}$;
\item\label{mijmkl} $M_{ji}M_{kl}U^{\al\be}(f_{ij}^\be f_{kl}^\al - f_{ik}^\be f_{lj}^\al) = U^{\al\nu} U^{\be\mu}  U^{\te\ga}\lan Df^\mu, Df_i^\ga \ran \lan Df^\nu, Df_k^\te \ran F^{\be\al}_{ik}$;
\end{enumerate}
where  
$F_{ik}^{\be\al}=f_i^\be f_k^\al - f_k^\be f_i^\al$ and $F^{\be\al}_{ik,l}=\frac{\p}{\p x_l}F^{\be\al}_{ik}$.
\end{lemma}
\begin{proof} Item \ref{dijdkl} follows from the fact that 
$f_{ii}^\be f_{kk}^\al - f_{ik}^\be f_{ik}^\al = (f_{i}^\be f_{kk}^\al - f_{k}^\be f_{ik}^\al)_i$.

Since $U^{\al\be}=U^{\be\al}$ we have that
\begin{eqnarray*}
\de_{kl}M_{ij}U^{\al\be}(f_{ij}^\be f_{kl}^\al - f_{ik}^\be f_{jl}^\al) &=& \de_{ij}M_{kl}U^{\al\be}(f_{kl}^\be f_{ij}^\al - f_{ki}^\be f_{lj}^\al)\\
&=& \de_{ij}M_{kl}(U^{\be\al}f_{kl}^\al f_{ij}^\be   - U^{\al\be}f_{jl}^\al f_{ik}^\be)\\
&=& \de_{ij}M_{kl}U^{\al\be}(f_{ij}^\be f_{kl}^\al  - f_{ik}^\be f_{jl}^\al),
\end{eqnarray*}
which proves Item \ref{dklmij}.

By using Item \ref{woodbury} of Lemma \ref{prelim1} we obtain 
\begin{eqnarray}
\de_{ij}M_{kl}U^{\al\be}(f_{ij}^\be f_{kl}^\al - f_{ik}^\be f_{jl}^\al) 
&=& f_k^\ga f_l^\mu U^{\ga\mu}U^{\al\be}(f_{ii}^\be f_{kl}^\al - f_{ik}^\be f_{il}^\al)
\nonumber\\
&=& U^{\al\be} U^{\ga\mu}(f_{ii}^\be f_k^\ga f_{kl}^\al f_l^\mu - f_{ki}^\be f_i^\ga f_{kl}^\al f_l^\mu)\nonumber\\
&=& U^{\ga\be} U^{\al\mu}f_{kl}^\ga f_l^\mu (f_{ii}^\be f_k^\al  - f_{ik}^\be f_i^\al)
\nonumber\\
&=& U^{\be\ga} U^{\al\mu}f_{ij}^\ga f_j^\mu (f_{kk}^\be f_i^\al  - f_{ki}^\be f_k^\al)
\nonumber\\
&=& U^{\al\ga} U^{\be\mu}\lan Df_i^\ga, Df^\mu \ran   (f_i^\be f_{kk}^\al   -  f_k^\be f_{ik}^\al). \label{DijMklUab}
\end{eqnarray}
Since $U^{\al\ga}U_{\ga\mu}=\de_{\al\mu}$ it follows that  
$U^{\al\ga}_i= - U^{\al\ga}U^{\be\mu}U_{\ga\mu, i}$, hence 
\begin{equation}\label{Ualbei}
U^{\al\ga}U^{\be\mu}\lan Df^\ga,Df^\mu_i\ran= -\,U^{\al\be}_i - U^{\al\ga}U^{\be\mu}\lan Df^\ga_i,Df^\mu\ran.
\end{equation}
It is easy to see that $F^{\be\al}_{ik,k}=(f_i^\be f_k^\al - f_k^\be f_i^\al)_k = (f_i^\be f_{kk}^\al - f_k^\be f_{ik}^\al) - (f_i^\al f_{kk}^\be - f_k^\al f_{ik}^\be)$. Thus we obtain
\begin{eqnarray}\label{interm}
U^{\al\ga}U^{\be\mu}\lan Df^\ga,Df^\mu_i\ran (f_i^\be f_{kk}^\al   -  f_k^\be f_{ik}^\al ) &=&  U^{\al\ga}U^{\be\mu}\lan Df^\ga,Df^\mu_i\ran F_{ik,k}^{\be\al}
\nonumber\\
&& +\,  U^{\al\ga}U^{\be\mu}\lan Df^\ga,Df^\mu_i\ran (f_i^\al f_{kk}^\be - f_k^\al f_{ik}^\be)
\nonumber\\
&=&  U^{\al\ga}U^{\be\mu}\lan Df^\ga,Df^\mu_i\ran F_{ik,k}^{\be\al}\\
&& +\,U^{\be\mu}U^{\al\ga}\lan Df^\mu,Df^\ga_i\ran (f_i^\be f_{kk}^\al - f_k^\be f_{ik}^\al).\nonumber
\end{eqnarray}
By using (\ref{Ualbei}) and (\ref{interm}) we obtain 
\begin{eqnarray}\label{exp2}
U^{\al\ga} U^{\be\mu}\lan Df_i^\ga, Df^\mu \ran (f_i^\be f_{kk}^\al   -  f_k^\be f_{ik}^\al) &=& -\, U^{\al\be}_i (f_i^\be f_{kk}^\al   -  f_k^\be f_{ik}^\al)\nonumber\\
&& - \,U^{\al\ga}U^{\be\mu}\lan Df^\ga,Df^\mu_i\ran (f_i^\be f_{kk}^\al   -  f_k^\be f_{ik}^\al)\nonumber\\
&=& -\, U^{\al\be}_i (f_i^\be f_{kk}^\al   -  f_k^\be f_{ik}^\al)\\ && 
- \, U^{\al\ga}U^{\be\mu}\lan Df^\ga,Df^\mu_i\ran F_{ik,k}^{\be\al} \nonumber\\
&& -\,U^{\be\mu}U^{\al\ga}\lan Df^\mu,Df^\ga_i\ran (f_i^\be f_{kk}^\al - f_k^\be f_{ik}^\al).\nonumber
\end{eqnarray}
Using (\ref{DijMklUab}) and (\ref{exp2}) we obtain\begin{eqnarray*}
2 \de_{ij}M_{kl}U^{\al\be}(f_{ij}^\be f_{kl}^\al - f_{ik}^\be f_{jl}^\al) &=& -\, U^{\al\be}_i (f_i^\be f_{kk}^\al   -  f_k^\be f_{ik}^\al)- U^{\al\ga}U^{\be\mu}\lan Df^\ga,Df^\mu_i\ran F^{\be\al}_{ik,k},
\end{eqnarray*}
which proves Item \ref{2dijmkl}.

Finally, again using Item \ref{woodbury} of Lemma \ref{prelim1}, we have 
\begin{eqnarray}
M_{ij}M_{kl}U^{\al\be}(f_{ij}^\be f_{kl}^\al - f_{ik}^\be f_{lj}^\al) &=& f_i^\ga f_j^\mu U^{\ga\mu} f_k^\te f_l^\nu U^{\te\nu} U^{\al\be}(f_{ij}^\be f_{kl}^\al - f_{ik}^\be f_{lj}^\al)\nonumber \\
&=& U^{\ga\mu} U^{\te\nu} U^{\al\be}\lan Df^\mu, Df_i^\be \ran \lan Df^\nu, Df_k^\al \ran f_i^\ga f_k^\te \nonumber \\&& -\, U^{\ga\mu} U^{\te\nu} U^{\al\be}\lan Df^\mu , Df_l^\al\ran \lan Df^\te,Df_i^\be \ran f_i^\ga f_l^\nu
\nonumber\\
&=& U^{\ga\mu} U^{\te\nu} U^{\al\be}\lan Df^\mu, Df_i^\be \ran \lan Df^\nu, Df_k^\al \ran f_i^\ga f_k^\te \nonumber\\&& -\, U^{\ga\mu} U^{\nu\te} U^{\be\al}\lan Df^\mu , Df_k^\be\ran \lan Df^\nu,Df_i^\al \ran f_i^\ga f_k^\te
\nonumber\\
&=& U^{\ga\mu} U^{\te\nu} U^{\al\be}\lan Df^\mu, Df_i^\be \ran \lan Df^\nu, Df_k^\al \ran F^{\ga\te}_{ik} 
\nonumber\\
&=& U^{\be\mu} U^{\al\nu} U^{\te\ga}\lan Df^\mu, Df_i^\ga \ran \lan Df^\nu, Df_k^\te \ran F^{\be\al}_{ik}.\nonumber
\end{eqnarray}
We conclude Item \ref{mijmkl}. Lemma \ref{prelim2} is proved.
\end{proof}

Now we will prove Proposition \ref{scalarcurvature-divX}. By using (\ref{scalar-expression1}), we have that 
$S= (\de_{ij}-M_{ij})(\de_{kl}-M_{kl})U^{\al\be}(f_{ij}^\be f_{kl}^\al - f_{ik}^\be f_{jl}^\al)$. Thus, by Lemma \ref{prelim2}, we obtain
\begin{eqnarray}\label{scalarexpression-4}
S&=& U^{\al\be}(f_i^\be f_{kk}^\al - f_k^\be f_{ik}^\al)_i
\nonumber \\&&
+U^{\al\be}_i (f_i^\be f_{kk}^\al -  f_k^\be f_{ik}^\al) + U^{\al\ga}U^{\be\mu}\lan Df^\ga,Df^\mu_i\ran F^{\be\al}_{ik,k}
\nonumber \\
&&+U^{\al\nu} U^{\be\mu}  U^{\te\ga}\lan Df^\mu, Df_i^\ga \ran \lan Df^\nu, Df_k^\te \ran F^{\be\al}_{ik}
\nonumber \\
&=& (U^{\al\be}(f_{i}^\be f_{kk}^\al - f_{k}^\be f_{ik}^\al)+ U^{\al\ga}U^{\be\mu}\lan Df^\ga,Df_k^\mu\ran F_{ki}^{\be\al})_i + V_{ik}^{\al\be}F_{ik}^{\be\al},
\end{eqnarray}
where $V_{ik}^{\al\be}$ is  given by
\begin{equation}\label{Vikalbe}
V_{ik}^{\al\be}=U^{\al\nu} U^{\be\mu}  U^{\te\ga}\lan Df^\mu, Df_i^\ga \ran \lan Df^\nu, Df_k^\te \ran -(U^{\al\ga} U^{\be\mu}\lan Df^\ga,Df_i^\mu\ran)_k.
\end{equation}
It holds that
\begin{lemma} $V_{ik}^{\al\be}F_{ik}^{\be\al}=\lan R^\perp(\na f^\ga,\na f^\mu)\eta^\mu,\eta^\ga\ran$.
\end{lemma}
In fact, by using (\ref{shape-oper}) it follows that 
\begin{eqnarray}\label{shape-comut}
g(A^\mu \p_i,A^\ga\p_k) &=& f_{il}^\mu g^{lr} f_{kr}^\ga = f_{il}^\mu f_{kr}^\ga (\de_{lr}-U^{\te\nu}f_l^\nu f_r^\te)\nonumber\\
&=& \lan Df_i^\mu, Df_k^\ga\ran - U^{\te\nu}\lan Df_i^\mu,Df^\nu\ran \lan Df_k^\ga,Df^\te\ran. 
\end{eqnarray}
Now, by (\ref{Vikalbe}) and using that $U^{\al\be}_r=-U^{\al\ga}U^{\be\mu}U_{\ga\mu,r}$, we obtain 
\begin{eqnarray*}
V_{ik}^{\al\be} 
&=& 
U^{\al\nu} U^{\be\mu}  U^{\te\ga}(U_{\mu\ga,i}-\lan Df_i^\mu, Df^\ga \ran)\lan Df^\nu, Df_k^\te \ran  - U^{\al\ga}_k U^{\be\mu}\lan Df^\ga,Df_i^\mu\ran 
\\
&& - U^{\al\ga} U^{\be\mu}_k \lan Df^\ga,Df_i^\mu\ran 
- U^{\al\ga} U^{\be\mu}\lan Df_k^\ga,Df_i^\mu\ran 
- U^{\al\ga} U^{\be\mu}\lan Df^\ga,Df_{ik}^\mu\ran
\\
&=& - U^{\al\nu}U^{\be\te}_i\lan Df^\nu, Df_k^\te \ran - U^{\al\nu} U^{\be\mu}  U^{\te\ga}\lan Df_i^\mu, Df^\ga \ran (U_{\nu\te,k}-\lan Df_k^\nu, Df^\te \ran) 
\\
&& - U^{\al\ga}_k U^{\be\mu}\lan Df^\ga,Df_i^\mu\ran 
- U^{\al\ga} U^{\be\mu}_k \lan Df^\ga,Df_i^\mu\ran 
\\
&& - U^{\al\ga} U^{\be\mu}\lan Df_k^\ga,Df_i^\mu\ran 
- U^{\al\ga} U^{\be\mu}\lan Df^\ga,Df_{ik}^\mu\ran
\\
&=& - C_{ik}^{\al\be} + U^{\al\ga}_k U^{\be\mu}\lan Df_i^\mu, Df^\ga \ran + U^{\al\nu} U^{\be\mu}  U^{\te\ga}\lan Df_i^\mu, Df^\ga\ran\lan Df_k^\nu, Df^\te \ran
\\
&& - U^{\al\ga}_k U^{\be\mu}\lan Df^\ga,Df_i^\mu\ran  
 - U^{\al\ga} U^{\be\mu}\lan Df_k^\ga,Df_i^\mu\ran 
\\
&=& -C_{ik}^{\al\be} + U^{\al\nu} U^{\be\mu}  U^{\te\ga}\lan Df_i^\mu, Df^\ga\ran\lan Df_k^\nu, Df^\te \ran
- U^{\al\ga} U^{\be\mu}\lan Df_k^\ga,Df_i^\mu\ran, 
\end{eqnarray*}
where $C_{ik}^{\al\be}$ is given by
\begin{equation*}
C_{ik}^{\al\be}= U^{\al\nu}U^{\be\te}_i\lan Df^\nu, Df_k^\te \ran + U^{\al\ga} U^{\be\mu}_k \lan Df^\ga,Df_i^\mu\ran + U^{\al\ga} U^{\be\mu}\lan Df^\ga,Df_{ik}^\mu\ran.
\end{equation*}
Note that $C_{ik}^{\al\be}=C_{ki}^{\al\be}$. Since $F_{ik}^{\be\al}=-F_{ki}^{\be\al}$ we obtain that  $C_{ik}^{\al\be}F_{ik}^{\be\al}=0$. Thus, by using that $\na f^\al=U^{\al\ga}f_i^\ga \p_i$, it follows from (\ref{shape-comut}) and Ricci's equation that
\begin{eqnarray}\label{Sperp-coord}
V_{ik}^{\al\be}F_{ik}^{\be\al} &=& - (f_i^\be f_k^\al - f_k^\be f_i^\al)U^{\al\ga}U^{\be\mu}g(A^\mu\p_i, A^\ga\p_k)
\\
&=& - (g(A^\mu(\na f^\mu), A^\ga(\na f^\ga)) - g(A^\mu (\na f^\ga), A^\ga(\na f^\mu)))\nonumber
\\
&=& - \lan R^\perp(\na f^\ga, \na f^\mu)\eta^\mu,\eta^\ga\ran, \nonumber
\end{eqnarray}
which together with (\ref{scalarexpression-4}) concludes the proof of Proposition \ref{scalarcurvature-divX}.

\section{Proof of Theorem \ref{mass-graph}.}
Since $f=(f^1,\ldots,f^m):\real^n\to \real^m$ is an asymptotically flat map it holds that  $f_i^\al = O(|x|^{-p/2})$ and $f_{ik}^\al = O(|x|^{-p/2-1})$, for all $i,k=1,\ldots,n$ and $\al=1,\ldots,m$. In particular, $U^{\al\ga}$ tends to $\de_{\al\ga}$ when $|x|\to \infty$. Furthermore, using \ref{woodbury} and \ref{U-inv-metric}, we have that $U^{\al\be}-\de_{\al\be}=-g(\na f^\al,f^\be)=U^{\al\ga}\lan Df^\ga,Df^\be\ran=O(|x|^{-p})$. This implies
\begin{equation*}
(U^{\al\be}-\de_{\al\be})(f_i^\be f_{kk}^\al - f_k^\be f_{ik}^\al) = O(|x|^{-2p-1})
\end{equation*}
and 
\begin{equation*}
U^{\al\ga}U^{\be\mu}\lan Df^\ga,Df_k^\mu\ran f_i^\al f_k^\be = O(|x|^{-2p-1}).
\end{equation*}
Since $p>(n-2)/2$ we have that $2p+1>n-1=\dim S_r$. Thus we obtain  
\begin{equation}\label{decaimass}
\lim_{r\to \infty}\int_{S_r} U^{\al\be}(f_i^\be f_{kk}^\al - f_k^\be f_{ik}^\al)\frac{x^i}{|x|}=\lim_{r\to \infty}\int_{S_r} (f_i^\al f_{kk}^\al - f_k^\al f_{ik}^\al)\frac{x^i}{|x|}
\end{equation}  
and
\begin{equation} \label{decai}
\lim_{r\to \infty}\int_{S_r} U^{\al\ga}U^{\be\mu}\lan Df^\ga,Df_k^\mu\ran (f_i^\al f_k^\be - f_k^\al f_i^\be)=0.
\end{equation}
Furthermore the function $S^\perp=\lan R^\perp(\na f^\al,\na f^\be)\eta^\be,\eta^\al\ran \in O(|x|^{-2p-2})$ since, by (\ref{Sperp-coord}) and (\ref{shape-comut}), it is expressed by 
\begin{equation}\label{S_perp}
S^\perp = U^{\al\ga} U^{\be\mu}(\lan Df_k^\ga,Df_i^\mu\ran + U^{\te\nu} \lan Df_i^\mu, Df^\nu\ran \lan Df_k^\ga, Df^\te \ran)(f_i^\al f_k^\be - f_i^\be f_k^\al).
\end{equation}
Since $2p+2>n$ it follows that $S^\perp:\real^n\to \real$ is integrable. We recall also that, by hypothesis, the scalar curvature function $S:\real^n\to \real$ is integrable.

Since $g_{kik} - g_{kki}= f_i^\al f_{kk}^\al - f_k^\al f_{ik}^\al$, it follows from Proposition \ref{scalarcurvature-divX} together with (\ref{decaimass}), (\ref{decai}) and the divergence theorem, that
\begin{eqnarray}\label{final}
\int_{R^n}S +S^\perp &=& \lim_{r\to \infty} \int_{S_r} \lan X,\frac{x}{|x|}\ran = \lim_{r\to \infty} \int_{S_r} (f_i^\al f_{kk}^\al - f_k^\al f_{ik}^\al)\frac{x^i}{|x|}\nonumber\\&=& 2(n-1)\om_{n-1} m_{ADM}\nonumber.
\end{eqnarray}
Theorem \ref{mass-graph} follows from the fact that  $\int_{R^n} S + S^\perp = \int_M (S + S^\perp)\frac{1}{\sqrt G} dM$.

\section{proof of Theorem \ref{penrose-graph}}

Let $\nu:\p\Om\to \real^n$ be the unit vector field orthogonal to $\p\Om$ pointing outward to $\Om$. Let $H^\Si=-\dv_{\real^n}\nu$ be the mean curvature of $\Si=\p\Om$ seen as a hypersurface of the Euclidean space $\real^n$.

Since each connected component of $\Si$ is a level set of $f^\al$, for all $\al$, it follows that the gradient vector field $Df^\al$ is normal to $\Si$, hence 
\begin{equation}\label{sign-LD}
Df^\al = \lan Df^\al, \nu\ran\,\nu \ \ \mbox{ in } \Si.
\end{equation}
In particular, $Df^\al$ and $Df^\be$ are linearly dependent which implies that
\begin{eqnarray}\label{LD}
f_i^\al f_k^\be-f_k^\al f_i^\be= \lan (Df^\be\wedge Df^\al)e_i,e_k \ran = 0 \ \ \mbox{ in } \Si,
\end{eqnarray} 
for all $\al,\be=1,\ldots,n$. Here, $"\wedge" :\real^n\times \real^n\to (\real^n)^*$ is the skew-symmetric tensor given by
$(u\wedge v)w=\lan v,w\ran u - \lan u,w\ran v$, for all $u,v,w\in \real^n$.

Using (\ref{expression-X}), (\ref{sign-LD}) and (\ref{LD})  we obtain
\begin{eqnarray}\label{inner}
\lan X, \nu \ran 
&=& U^{\al\be}(f_i^\be f_{kk}^\al - f_k^\be f_{ik}^\al) \nu^i \\
&=& U^{\al\be}(\De f^\al \lan Df^\be,\nu\ran - \hs_{f^\al}(Df^\be,\nu))\nonumber
\end{eqnarray}
By a simple computation  we have that 
\begin{eqnarray}\label{laplac}
\De f^\al &=& \De_\Si f^\al + \hs_{f^\al} (\nu,\nu) - H^\Si \lan  \nu ,Df^\al\ran 
\end{eqnarray}
Using that $f^\al$ is constant along $\Si$ it follows that $\De_\Si f^\al\equiv 0$ and $Df^\be = \lan Df^\be,\nu\ran \nu$ in $\Si$. Thus, by (\ref{inner}) and (\ref{laplac}), 
we obtain 
\begin{eqnarray}\label{inner2}
\lan X, \nu \ran &=& U^{\al\be}(\hs_{f^\al} (\nu,\nu)\lan Df^\be,\nu\ran - \hs_{f^\al}(Df^\be,\nu) - H^\Si \lan \nu ,Df^\al\ran\lan Df^\be,\nu\ran ) \nonumber
\\&=& -\, U^{\al\be}H^\Si \lan \nu ,Df^\al\ran\lan Df^\be,\nu\ran = -\, U^{\al\be}H^\Si \lan Df^\al, Df^\be\ran 
\end{eqnarray}
%
Again using (\ref{sign-LD}), we write $Df^\al = \la^\al \nu$, in $\Si$, where $\la^\al=\lan Df^\al,\nu \ran$. In particular, $U_{\al\be}=\de_{\al\be}+\la^\al \la^\be$, in $\Si$, for all $\al,\be$. This implies that $U^{\al\be}=\de_{\al\be} - \la^\al\la^\be/(1+|\la^2|)$, in $\Si$, where $|\la|^2=|Df|^2=(\la^1)^2+\ldots + (\la^m)^2$.
Thus, in $\Si$, it holds 
\begin{eqnarray}\label{inverseSigma}
U^{\al\ga}\lan Df^\ga, Df^\al\ran &=& U^{\al\ga}\la^\al\la^\ga = (\de_{\al\be}-\frac{\la^\al\la^\be}{1+|\la|^2})\la^\al\la^\be 
= \frac{|Df|^2}{1+|Df|^2}.
\end{eqnarray}

As in the proof of Theorem \ref{mass-graph}, using that $f$ is an asymptotically flat map we have that $\lim_{r\to \infty}\int_{S_r} \lan X,\frac{x}{|x|}\ran = 2(n-1)\om_{n-1} m_{ADM}$.
By Proposition \ref{scalarcurvature-divX} and  the divergence theorem, we obtain from (\ref{inner2}) and (\ref{inverseSigma}) that 
\begin{eqnarray}
\int_{\real^n-\Om} S+S^\perp &=& \lim_{r\to \infty}\int_{S_r} \lan X,\frac{x}{|x|}\ran + \int_\Si \lan X,\nu\ran \label{pen_afinal}\\ 
&=& 2(n-1)\om_{n-1} m_{ADM} - \int_\Si \frac{|Df|^2}{1+|Df|^2} H^\Si. \nonumber
\end{eqnarray}
Theorem \ref{penrose-graph} is proved.

\section{proof of Theorem \ref{cor-penrose-graph}}
Let $A:\real^m \to \real^m$ be an isometry that transforms the straight line $\ell$ into the vertical line $A(\ell)=\{(t,0,\ldots,0)\in \real^m\bigm| t\in \real\}$. Consider the map $\bar f=A\circ f:\real^n\setminus\Om\to \real^m$ and let $M_A$ be the graph of $\bar f$ with its natural metric. 
Since $\bar\vp(x)=(x,\bar f(x))=(x,f^\al(x) \bar e_\al)$, where $\bar e_\al=A e_\al$, we obtain that $M_A$ is isometric to $M$. This implies that the ADM mass  and scalar curvature  of $M_A$ coincides with $m_{ADM}$ and $S$, respectively. Furthermore, by (\ref{S_perp}), we also obtain that the normal function $S_A^\perp$ (as defined in (\ref{normal-scalar})) coincides with $S^\perp$. Thus, without loss of generality, we can assume  that the straight line $\ell=\real\times \{0\}\subset \real^m$.

First we claim the following
\begin{claim}\label{cl-cylinder} For all $\ga\in \{1,\ldots,m\}$ and $x_0\in \Si$ it holds that $\lim_{x\to x_0}\na f^\ga(x) = \pm(0,\de_{1\ga} e_1)$. In particular, $\lim_{x\to \Si}U^{\al\be}=\de_{\al\be}-\de_{1\al}\de_{1\be}$, for all  $\al,\be$.
\end{claim} 
In fact, we fix $x_0\in \Si$. First we assume, by contradiction, that $\lim_{x\to x_0}\na f^\al=0$, for all $\al$. By Item \ref{U-inv-metric} of Lemma \ref{prelim1}  we have that $U^{\al\be}=\de_{\al\be}-g(\na f^\al,\na f^\be)$. Thus we obtain 
\begin{equation}\label{contU}
\lim_{x\to x_0}U^{\al\be}=\de_{\al\be},
\end{equation}
for all $\al,\be$. Using that $\p_i=(e_i,f_i^\be e_\be)$, $\na f^\ga = U^{\ga\al}f_i^\al \p_i$ and $g(\na f^\ga,\na f^\be)=U^{\ga\al}\lan Df^\al,Df^\be\ran$ we have 
\begin{eqnarray}\label{grad-g}
\na f^\ga &=& U^{\ga\al} f_i^\al \p_i = U^{\ga\al}(Df^\al,\lan Df^\al,Df^\be\ran e_\be)\nonumber\\ 
&=& (U^{\ga\al} Df^\al, \,g(\na f^\ga,\na f^\be) e_\be).
\end{eqnarray}
Consider $\pi^1,\pi^2:\real^n\times \real^m\to \real^n$ the orthogonal projections $\pi^1(x,y)=x$ and $\pi^2(x,y)=y$. By (\ref{contU}) and (\ref{grad-g}) we have  $$0=\lim_{x\to x_0}\pi^1_*(\na f^\al) = \lim_{x\to x_0}U^{\al\be}Df^\be=\lim_{x\to x_0}Df^\al,$$ for all $\al$. Thus the graph of $M$ is tangent to the plane $\real^n\times \{0\}$, along $\p M$. This is a contradiction.  Thus, we can set $1\le\ga\le m$  so that $\limsup_{x\to x_0}|\na f^\ga|>0$. 

Using that $M$ is tangent to the cylinder $\Si\times (\real\times \{0\})$ along $\Si=\p M$ we obtain that the vector $\eta=(0,e_1)$ is tangent to $M$ and normal to $\p M$, hence it is a conormal vector field along $\p M$.  Since each connected component of $\Si$ is  a level set of $f^\al$, for all $\al$, and $M$ extends $C^2$ up its boundary, it is easy to see that, for all $\al$, either $\lim_{x\to x_0} |\na f^\al|=0$ or $\lim_{x\to x_0} \na f^\al/|\na f^\al|=\pm\eta$. By using (\ref{grad-g}), that $\eta=(0,e_1)$ and that $\lim_{x\to x_0} \na f^\ga/|\na f^\ga|=\pm\eta$ we obtain
\begin{equation}\label{e1}
\pm e_1 = \lim_{x\to x_0} \pi^2_*(\na f^\ga/|\na f^\ga|)
= \lim_{x\to x_0}g(\na f^\ga/|\na f^\ga|,\na f^\be) e_\be.
\end{equation} 
This implies that 
\begin{equation}\label{si1}
\begin{array}{l}
\lim_{x\to x_0}g(\na f^\ga/|\na f^\ga|,\na f^\be)=0, \mbox{ for all } \be\neq 1;\\
\lim_{x\to x_0}g(\na f^\ga/|\na f^\ga|,\na f^1)=\pm 1.
\end{array}
\end{equation}
If we assume that $\limsup_{x\to x_0}|\na f^\be|>0$, for some $\be\neq 1$ then, by (\ref{si1}), we obtain
\begin{eqnarray*}
0&=& \lim_{x\to x_0}g({\na f^\ga}/{|\na f^\ga|},\na f^\be) = \limsup_{x\to x_0} g\big({\na f^\ga}/{|\na f^\ga|},{\na f^\be}/{|\na f^\be|}\big)|\na f^\be| \\&=&\pm g(\eta,\eta)\limsup_{x\to x_0}|\na f^\be| =\pm\limsup_{x\to x_0}|\na f^\be|,
\end{eqnarray*}
which is a contradiction. Thus, it holds that 
\begin{equation}\label{resu1}
\lim_{x\to\Si}\na f^\be=0, \mbox{ for all } \be\neq 1,
\end{equation}
We conclude that $\ga=1$, which implies that $\pm\eta = \lim_{x\to x_0}\na f^1/|\na f^1|$. Moreover, again using (\ref{si1}), we obtain that $\lim_{x\to \Si}|\na f^1|=\lim_{x\to \Si}g(\na f^1/|\na f^1|,\na f^1)=1$. This implies that $\lim_{x\to x_0}\na f^1=\pm (0,e_1)$. Claim \ref{cl-cylinder} is proved. 

The following claim follows as a consequence of Claim \ref{cl-cylinder}.
\begin{claim}\label{lim_infty} $\lim_{x\to\Si}  Df^\al = 0$, for all  $\al\neq 1$, and $\lim_{x\to \Si} |Df^1|=+\infty$.
\end{claim}
In fact, it follows from (\ref{grad-g}) together with Claim \ref{cl-cylinder}  that $\lim_{x\to \Si}U^{\ga\al}Df^\al = 0$, for all  $\ga$. Thus, since  $1=\lim_{x\to \Si}g(\na f^1,\na f^1)=\lim_{x\to \Si}U^{1\al}\lan Df^\al,Df^1\ran$ we have that $\lim_{x\to \Si} |Df^1|=+\infty$. Claim \ref{lim_infty} is proved.   

Let $F^k=(f^{1;k},f^{2;k},\ldots,f^{m;k}):\real^n\setminus\Om\to \real^m$, with $k=1,2,\ldots$, be a sequence of smooth maps satisfying:
\begin{enumerate}[(i)]
\item\label{coincides} $F^{k}$ coincides with $f$ outside a compact subset containing $\Si$;
\item\label{supp-compact} $F^k=0$ everywhere in $\Si$ and moreover the map $(f^{2;k},\ldots,f^{m;k})$ vanishes in a neighborhood $U_k\subset \real^n\setminus\Om$ of the boundary $\p\Om$ with $\m L^n$-measure $|U_k|\to 0$, when $k\to \infty$;
\item\label{convergesC2} if $M_k$ is the graph of $f_k$ with its natural metric then the closure $\bar M_k$ converges to $\bar M$ with respect to the $C^2$-topology.
\end{enumerate}
Note that Theorem \ref{penrose-graph} applies for $F^k:\real^n\setminus\Om\to \real^m$. By using \ref{coincides}, the ADM-mass of $M_k$ coincides with the ADM-mass of $M$. By using \ref{convergesC2} and (\ref{scalar-expression1}), for all $x\in \real^n\setminus\Om$, the scalar curvature $S_k:\real^n\setminus\Om\to \real$ of the graph $\bar M_k$ satisfies $\lim_{k\to\infty} S_k(x)=S(x)$. By using (\ref{S_perp}), we have that $S_k^\perp$ converges to $S^\perp$ in $\real^n\setminus\bar\Om$. By using \ref{supp-compact}, the function $S^\perp_k$  vanishes in a neighborhood of $\Si=\p\Om$, hence $S^\perp_k$ also converge to $S^\perp$ in the points of $\Si$. We obtain that $S_k+S^\perp_k$ converges uniformly to $S+S^\perp$ in $\real^n\setminus\Om$. Thus, by Theorem \ref{penrose-graph}, we have
\begin{eqnarray*}
m_{ADM}&=&\frac{1}{2(n-1)\om_{n-1}}\lim_{k\to\infty}\left(\int_{\real^n\setminus\Om}(S_k+S^\perp_k) + \int_\Si \frac{|DF^k|^2}{1+|DF^k|^2}H^\Si\right)\\
&=& \frac{1}{2(n-1)\om_{n-1}}\left(\int_{\real^n\setminus\Om}(S+S^\perp) + \int_\Si H^\Si\right)
\end{eqnarray*}
Theorem \ref{cor-penrose-graph} is proved.

\section*{Aknowledgement}
The author are very grateful to Professor Fernando Cod\'a Marques for your suggestions and comments. 

\vspace{1.5cm}
\begin{small}
\begin{tabular}{lcr}

\begin{tabular}{l}
Heudson Mirandola\\
Universidade Federal do Rio de Janeiro\\
Instituto de Matem\'{a}tica\\
21945-970 Rio de Janeiro-RJ\\
Brazil\\
\verb+mirandola@im.ufrj.br+\\
\end{tabular}\
&\quad&
\begin{tabular}{l}
Feliciano Vit\'orio\\
Universidade Federal de Alagoas\\
Instituto de Matem\'{a}tica\\
57072-900 Macei\'o-AL\\
Brazil\\
\verb+feliciano@pos.mat.ufal.br+\\
\end{tabular}\

\end{tabular}
\end{small}
\end{document}